\documentclass[reqno]{amsart}

\usepackage{enumitem}
\usepackage{amssymb,amsmath}
\usepackage{hyperref}
\usepackage{microtype}
\usepackage{bbm}

\newcommand*{\mailto}[1]{\href{mailto:#1}{\nolinkurl{#1}}}
\newcommand{\arxiv}[1]{\href{http://arxiv.org/abs/#1}{arXiv:#1}}

\newcommand{\msc}[1]{\href{http://www.ams.org/msc/msc2010.html?t=&s=#1}{#1}}

\newtheorem{theorem}{Theorem}[section]
\newtheorem{lemma}[theorem]{Lemma}
\newtheorem{corollary}[theorem]{Corollary}
\newtheorem{remark}[theorem]{Remark}

\newcommand{\R}{{\mathbb R}}

\newcommand{\C}{{\mathbb C}}

\newcommand{\id}{{\mathbbm{1}}}
\newcommand{\OO}{{\mathcal O}}
\newcommand{\cW}{{\mathcal W}}

\newcommand{\loc}{\mathrm{loc}}

\newcommand{\dom}{\mathrm{dom}}

\newcommand{\I}{\mathrm{i}}
\newcommand{\E}{\mathrm{e}}

\DeclareMathOperator{\im}{Im}

\newcommand{\nn}{\nonumber}
\newcommand{\be}{\begin{equation}}
\newcommand{\ee}{\end{equation}}

\newcommand{\ti}{\tilde}
\newcommand{\abs}[1]{\left\lvert #1 \right\rvert}

\newcommand{\norm}[1]{\left\lVert #1 \right\rVert}

\newcommand{\eps}{\varepsilon}

\newcommand{\lam}{\lambda}

\numberwithin{equation}{section}

\begin{document}

\title[Dispersion Estimates: The Effect of Boundary Conditions]{Dispersion Estimates for Spherical Schr\"odinger Equations: The Effect of Boundary Conditions}

\dedicatory{Dedicated with great pleasure to Petru A. Cojuhari on the occasion of his 65th birthday}

\author[M. Holzleitner]{Markus Holzleitner}
\address{Faculty of Mathematics\\ University of Vienna\\
Oskar-Morgenstern-Platz 1\\ 1090 Wien\\ Austria}
\email{\mailto{amhang1@gmx.at}}

\author[A. Kostenko]{Aleksey Kostenko}
\address{Faculty of Mathematics\\ University of Vienna\\
Oskar-Morgenstern-Platz 1\\ 1090 Wien\\ Austria}
\email{\mailto{duzer80@gmail.com};\mailto{Oleksiy.Kostenko@univie.ac.at}}
\urladdr{\url{http://www.mat.univie.ac.at/~kostenko/}}

\author[G. Teschl]{Gerald Teschl}
\address{Faculty of Mathematics\\ University of Vienna\\
Oskar-Morgenstern-Platz 1\\ 1090 Wien\\ Austria\\ and International 
Erwin Schr\"odinger Institute for Mathematical Physics\\ 
Boltzmanngasse 9\\ 1090 Wien\\ Austria}
\email{\mailto{Gerald.Teschl@univie.ac.at}}
\urladdr{\url{http://www.mat.univie.ac.at/~gerald/}}

\thanks{{\it Research supported by the Austrian Science Fund (FWF) 
under Grants No.\ P26060 and W1245}}

\thanks{Opuscula Math. {\bf 36}, no.~6, 769--786 (2016)}

\begin{abstract}
We investigate the dependence of the $L^1\to L^\infty$ dispersive estimates for one-dimensional radial Schr\"o\-din\-ger operators on boundary conditions at $0$. In contrast to the case of additive perturbations, we show that the change of a boundary condition at zero results in the change of the dispersive decay estimates if the angular momentum is positive, $l\in (0,1/2)$. However, for nonpositive angular momenta, $l\in (-1/2,0]$,  the standard $O(|t|^{-1/2})$ decay remains true for all self-adjoint realizations.
\end{abstract}

\maketitle

\section{Introduction}

We are concerned with the one-dimensional Schr\"odinger equation
\begin{equation} \label{Schr}
  \I \dot \psi(t,x) = H_\alpha \psi(t,x), \quad 
  H_\alpha:= - \frac{d^2}{dx^2} + \frac{l(l+1)}{x^2},\quad 
  (t,x)\in\R\times \R_+,
\end{equation}
with the angular momentum 
$|l| < \frac{1}{2}$ and self-adjoint boundary conditions at $x=0$ parameterized by a parameter $\alpha \in [0,\pi)$ (the definition is given in Section \ref{sec:II},
see \eqref{eq:Gamma01}--\eqref{eq:bc_alpha} --- for recent discussion of this family of operators see \cite{ab,dr}). 
More precisely, we are interested in the dependence of the $L^1\to L^\infty$ dispersive estimates associated to the evolution group $\E^{-\I tH_\alpha}$ on the parameters $\alpha\in [0,\pi)$ and $l\in (-1/2,1/2)$.

On the whole line such results have a long tradition and we refer to Weder \cite{wed}, Goldberg and Schlag \cite{GS}, Egorova, Kopylova, Marchenko and Teschl \cite{EKMT}, as well as the reviews \cite{K10,schlag}. On the half line, the case $l=0$ with a Dirichlet boundary condition was treated by Weder \cite{wed2}. The case of general $l$ and the Friedrichs boundary condition at $0$ ($\alpha=0$ in our notation)
\be\label{eq:bc}
\lim_{x\to0} x^l ( (l+1)f(x) - x f'(x))=0, \qquad l\in\Big(-\frac{1}{2},\frac{1}{2}\Big),
\ee
was recently considered in Kova\v{r}\'{i}k and Truc \cite{kotr} and they proved (see Theorem 2.4 in \cite{kotr}) that 
\be\label{eq:01}
\|\E^{-\I tH_0}\|_{L^1(\R_+)\to L^\infty(\R_+)} = \OO(|t|^{-1/2}),\quad t\to\infty.
\ee
It was proved in \cite{ktt} that this estimate remains true under additive perturbations. More precisely (see \cite[Theorem 1.1]{ktt}), let $H=H_0+q$, where the potential $q$ is a real integrable on $\R_+$ function. If in addition
\be\label{eq:q_hyp}
\int_0^1 |q(x)| dx <\infty \quad \text{and} \quad  
\int_1^\infty x^{\max(2,l+1)} |q(x)| dx <\infty,
\ee
and there is neither a resonance nor an eigenvalue at $0$, then 
\begin{equation}\label{full}
\norm{\E^{-\I tH}P_c(H)}_{L^1(\R_+)\to L^\infty(\R_+)}
	= \mathcal{O}(|t|^{-1/2}),
\quad t\to\infty.
\end{equation}
Here $P_c(H)$ is the orthogonal projection in $L^2(\R_+)$ onto 
the continuous spectrum of $H$.

The main result of the present paper shows that the decay estimates \eqref{eq:01} and \eqref{full} are no longer true for $\alpha\in (0,\pi)$ if $l\in(0,1/2)$. In other words, this means that singular rank one perturbations destroy these decay estimates if $l\in (0,1/2)$ (since the change of a boundary condition can be considered as a rank one perturbation in the resolvent sense). Namely, consider first the operator $H_{\pi/2}$, which is associated with the following boundary condition at $x=0$:
\be\label{eq:02}
\lim_{x\to0} x^{-l-1} ( lf(x) + x f'(x))=0, \qquad l\in\Big(-\frac{1}{2},\frac{1}{2}\Big).
\ee
 
\begin{theorem}\label{th:pi/2}
Let $|l|<1/2$. Then 
\be\label{eq:03a}
\|\E^{-\I t H_{\pi/2}}\|_{L^1(\R_+) \to L^\infty(\R_+)} = \OO(|t|^{-1/2}), \quad t\to \infty,
\ee
for all $l\in (-1/2, 0]$, and 
\be\label{eq:03b}
\|\E^{-\I t H_{\pi/2}}\|_{L^1(\R_+,  \max(x^{-l},1)) \to L^\infty(\R_+, \min(x^l,1))} = \OO(|t|^{-1/2+l}), \quad t\to \infty,
\ee
whenever $l\in (0,1/2)$. The last estimate is sharp.
\end{theorem}

In the remaining case $ \alpha\in (0,\pi/2)\cup (\pi/2,\pi)$, the decay estimate is given by the the next theorem.

\begin{theorem}\label{th:no3}
Let $|l|<1/2$ and $ \alpha\in (0,\pi/2)\cup (\pi/2,\pi)$. Then 
\be\label{eq:04a}
\|\E^{-\I t H_{\alpha}}P_c(H_\alpha)\|_{L^1(\R_+) \to L^\infty(\R_+)} = \OO(|t|^{-1/2}), \quad t\to \infty,
\ee
for all $l\in (-1/2, 0]$, and 
\be\label{eq:04b}
\|\E^{-\I t H_{\alpha}}P_c(H_\alpha)\|_{L^1(\R_+,   \max(x^{-l},1)) \to L^\infty(\R_+,\min(x^l,1))} = \OO(|t|^{-1/2}), \quad t\to \infty,
\ee
whenever $l\in (0,1/2)$.
\end{theorem}

Notice that in the case $l\in(0,1/2)$ we need to consider weighted $L^1$ and $L^\infty$ spaces since functions contained in the domain of $H_\alpha$ might be unbounded near $0$.

Finally, let us briefly outline the content of the paper. In the next section we define the operator $H_\alpha$ and collect its basic spectral properties.
 Section \ref{sec:III} contains the proof of Theorem \ref{th:pi/2}. In particular, we compute explicitly the kernel of the evolution group $\E^{-\I tH_{\pi/2}}$ and this enables us to prove  \eqref{eq:03a} and \eqref{eq:03b} by using the estimates for Bessel functions $J_\nu$ (all necessary facts on Bessel functions are contained in Appendix \ref{sec:Bessel}). Theorem \ref{th:no3} is proved in Section \ref{sec:IV}. Its proof is based on the use of a version of the van der Corput lemma, which is given in Appendix \ref{sec:vdCorput}. Also Appendix \ref{sec:vdCorput} contains necessary facts about the Wiener algebras $\cW_0(\R)$ and $\cW(\R)$. 
In the final section we formulate some sufficient conditions for a function $f(H)$ of a 1-D Schr\"odinger operator $H$ to be an integral operator.

\section{Self-adjoint realizations and their spectral properties}\label{sec:II}

Let $l\in (-1/2,1/2)$ and denote by $H_{\max}$ the maximal operator associated with 
\[
\tau=- \frac{d^2}{dx^2} + \frac{l(l+1)}{x^2}
\]
 in $L^2(\R_+)$. Note that $\tau$ is limit point at infinity and limit circle at $x=0$ since $|l|<1/2$. Therefore, self-adjoint restrictions of $H_{\max}$ (or in other words, self-adjoint realizations of $\tau$ in $L^2(\R_+)$) form a 1-parameter family. More precisely (see, e.g., \cite{ek} and also \cite{ab}), the following limits
\be\label{eq:Gamma01}
\Gamma_0 f:=  \lim_{x\to 0} W_x(f, x^{l+1}),\quad \Gamma_1 f:=  \frac{-1}{2l+1}\lim_{x\to 0} W_x(f, x^{-l})
\ee
exist and are finite for all $f\in \dom(H_{\max})$. 
Self-adjoint restrictions $H_\alpha$ of $H_{\max}$ are parameterized by the following boundary conditions at $x=0$:
\be\label{eq:bc_alpha}
\dom(H_\alpha) = \{f\in \dom(H_{\max})\colon \sin(\alpha)\, \Gamma_1 f = \cos(\alpha)\, \Gamma_0 f\},\quad \alpha\in [0,\pi).
\ee
Note that the case $\alpha=0$ corresponds to the Friedrichs extension of $H_{\min} = H_{\max}^\ast$. 

Let $\phi(z,x)$ and $\theta(z,x)$ be the fundamental system of solutions of $\tau u = z u$ given by
\begin{align}\label{eq:phi}
\begin{split}
\phi(z,x) &= C_l^{-1} \sqrt{\frac{\pi x}{2}}{z^{-\frac{2l+1}{4}}} 
		J_{l+\frac{1}{2}}(\sqrt{z} x),\\
\theta(z,x) &= C_l \sqrt{\frac{\pi x}{2}}\frac{z^{\frac{2l+1}{4}}}{\sin((l\!+\!\frac{1}{2})\pi)}
		J_{-l-\frac{1}{2}}(\sqrt{z} x),
\end{split}
\end{align}
where $J_\nu$ is the Bessel function of order $\nu$ (see Appendix \ref{sec:Bessel}) and
\be
C_l = \frac{\sqrt{\pi}}{\Gamma(l+\frac{3}{2}) 2^{l+1}}.
\ee
The Weyl solution normalized by $\Gamma_0\psi = 1$ is given by
\be
\psi(z,x) = \theta(z,x) + m(z)\phi(z,x) = C_l\I z^{\frac{2l+1}{4}} \sqrt{\frac{\pi x}{2}} H^{(1)}_{l+1/2} (\sqrt{z}x) \in L^2(0,\infty),
\ee
where $H_\nu^{(1)}$ is the Hankel function of the first kind \cite[Chapter X.2]{dlmf}, and 
\be\label{eq:m}
m(z) = -C_l^{2}\frac{(-z)^{l+1/2}}{\sin((l+\frac{1}{2})\pi)}  , \quad z\in\C\setminus\R_+,
\ee
is the Weyl function associated with $H_0$. Here the branch cut of the root is taken along the negative real axis. Notice that
\be\label{eq:rho}
d\rho(\lam) = \frac{C_l^2}{\pi} \id_{[0,\infty)}(\lam)\lam^{l+\frac{1}{2}}d\lam
\ee
is the corresponding spectral measure. 
It follows from \eqref{eq:Jnu01} that
\[
\phi(z,x) = x^{l+1}(1+o(1)), \quad \theta(z,x) = \frac{x^{-l}}{2l+1}(1+o(1)),
\] 
as $x\to 0$ and, moreover, 
\[
\Gamma_0 \theta = \Gamma_1 \phi= 1,\quad  \Gamma_1 \theta = \Gamma_0 \phi= 0.
\]
Set
\be\label{eq:phi_a}
\begin{split}
\phi_\alpha(z,x) &:= \cos(\alpha) \phi(z,x) + \sin(\alpha) \theta(z,x),\\
 \theta_\alpha(z,x) &:= \cos(\alpha) \theta(z,x) - \sin(\alpha) \phi(z,x),
 \end{split}
\ee
for all $z\in\C$. Therefore, $W(\theta_\alpha,\phi_\alpha) = 1$ and 
\be\label{eq:psi_a}
\psi_\alpha(z,x) := \theta_\alpha(z,x) + m_\alpha(z)\phi_\alpha(z,x), \quad m_\alpha(z)=\frac{m(z)\cos(\alpha) + \sin(\alpha)}{\cos(\alpha) - m(z) \sin(\alpha)},
\ee
is a Weyl solution normalized by $W(\psi_\alpha,\phi_\alpha) = 1$. Hence
\be
G_\alpha(z;x,y) = \begin{cases} \phi_\alpha(z,x) \psi_\alpha(z,y), & x\le y, \\  \phi_\alpha(z,x) \psi_\alpha(z,y), & x\ge y, \end{cases}
\ee
is the Green's function of $H_\alpha$. The absolutely continuous spectrum remains unchanged, $\sigma_{\rm ac}(H_\alpha) = [0,\infty)$, but there is one additional eigenvalue
\be
E_\alpha = -\left(\frac{\cot(\alpha) \cos(l\pi)}{C_l^{2}}\right)^{\frac{2}{2l+1}}
\ee
if $\frac{\pi}{2}<\alpha<\pi$. Finally, since 
\be\label{eq:im_m_a}
\im m_\alpha(z) = \frac{\im m(z)}{|\cos(\alpha) - m(z) \sin(\alpha)|^2}, 
\ee
we get the absolutely continuous part of the corresponding spectral measure of the operator $H_\alpha$:
\begin{align}\label{eq:rho_a}
\begin{split}
\rho_{\alpha}'(\lambda)d\lam &= \frac{1}{\pi} \im m_\alpha(\lam+\I0)d\lam\\
&=\frac{1}{\pi}\frac{C_l^2\lambda^{l+1/2}\id_{[0,\infty)}(\lam)}{(\cos(\alpha) - C_l^2 \sin(\alpha)\tan(\pi l)\lam^{l+1/2})^2 + C_l^4\sin^2(\alpha) \lam^{2l+1}} d\lam. 
\end{split}
\end{align}

\section{Proof of Theorem \ref{th:pi/2}}\label{sec:III}

Similar to the case $\alpha=0$ (see \cite{kotr}), the kernel of the evolution group $\E^{-\I tH_{\pi/2}}$ can be computed explicitly. 

\begin{lemma}\label{lem:pi/2}
Let $|l|<1/2$. Then the evolution group $\E^{-\I tH_{\pi/2}}$ is an integral operator for all $t\neq 0$ and its kernel is given by
\be\label{eq:evkern2}
[\E^{-\I tH_{\pi/2}}](x,y)
       = \frac{\I^{l-1/2}}{2 t} \E^{\I \frac{x^2 + y^2}{4t}} \sqrt{xy}\, J_{-l-1/2}\left(\frac{xy}{2 t}\right),
\ee
for all $x$, $y>0$ and $t\neq 0$.
\end{lemma}

\begin{proof}
First, notice that
\[
\phi_{\pi/2}(z,x) = \theta(z,x), \quad m_{\pi/2}(z) = -1/m(z),
\]
and then define the spectral transformation $U\colon L^2(\R_+)\to L^2(\R_+;\rho_{\pi/2})$ by
\[
U\colon f\mapsto \hat{f},\quad \hat{f}(\lambda):= \int_{\R_+} \theta(\lambda,x) f(x)dx,
\]
for every $f\in L^2_c(\R_+)$. Notice that $U$ extends to an isometry on $L^2(\R_+)$ and its inverse $U^{-1}\colon L^2(\R_+;\rho_{\pi/2})\to L^2(\R_+)$ is given by
\[
U^{-1}\colon g\mapsto \check{g}, \quad \check{g}(x):= \int_{\R_+} \theta(\lambda,x) g(\lambda)d\rho_{\pi/2}(\lambda),
\]
for all $g\in L^2_c(\R_+;\rho_{\pi/2})$. 
Therefore, we get by using \eqref{eq:phi} and \eqref{eq:rho_a} 
\begin{align}
(\E^{-(\I t+\varepsilon) H_{\pi/2}} f)(x) =&  (U^{-1} \E^{-(\I t+\varepsilon)\lambda} Uf)(x) = (U^{-1} \E^{-(\I t+\varepsilon)\lambda} \check{f})(x)\nn\\
&=\int_{\R_+} \theta(\lambda,x) \E^{-(\I t+\varepsilon)\lambda} \int_{\R_+} \theta(\lambda,y) f(y)\,dy\, d\rho_{\pi/2}(\lambda) \nn\\
 &=  \int_{\R_+} \int_{\R_+} \E^{-(\I t+\varepsilon) \lam} \frac{\sqrt{xy}}{2} J_{-l-\frac{1}{2}}(\sqrt{\lam} x) J_{-l-\frac{1}{2}}(\sqrt{\lam} y) f(y)\, dy\, d\lam.\nn
\end{align}
Since $|l|<1/2$, \eqref{eq:Jnu01} implies that
\be\label{eq:3.11}
|J_{-l-1/2}(k)| \le \frac{2^{l+1/2}}{\Gamma(1/2-l) k^{l+1/2}} (1 +\OO(k))
\ee
as $k\to 0$. Noting that $f\in L^2_c(\R_+)$ and using \eqref{eq:3.11}, Fubini's theorem implies 
\begin{align}
(\E^{-(\I t+\varepsilon) H_{\pi/2}} f)(x) =  \int_{\R_+} f(y) \int_{\R_+} \E^{-(\I t+\varepsilon) \lam} \frac{\sqrt{xy}}{2} J_{-l-\frac{1}{2}}(\sqrt{\lam} x) J_{-l-\frac{1}{2}}(\sqrt{\lam} y)  d\lam\, dy.
\end{align}
The integral 
\begin{align} \label{eq:evkern}
[\E^{-(\I t + \varepsilon) H_{\pi /2}}](x,y):
	= \frac{\sqrt{xy}}{2} \int_{0}^{\infty} \E^{-\I t \lam} J_{-l-\frac{1}{2}}(\sqrt{\lam} x) J_{-l-\frac{1}{2}}(\sqrt{\lam} y) d\lam
\end{align}
is known as Weber's second exponential integral \cite[\S13.31]{Wat} (cf.\ also \cite[(4.14.39)]{erd}) and hence
\[
(\E^{-(\I t+ \varepsilon) H_{\pi /2}} f)(x)=\frac{1}{\varepsilon+ \I t}\int_{0}^{\infty}  \E^{-\frac{x^2 + y^2}{4(\varepsilon+ \I t)}} \frac{\sqrt{xy}}{2} I_{-l-\frac{1}{2}}\Big(\frac{xy}{2(\varepsilon+\I t)}\Big) f(y)dy,
\]
where $I_\nu$ is the modified Bessel function (see \cite[Chapter X]{dlmf} and in particular formula (10.27.6) there)
\be \label{bessel2nd}
I_\nu(z) = \sum_{n=0}^\infty \frac{(z/2)^{\nu+2n}}{n! \Gamma(\nu +m+1)} =\E^{\mp \I\nu\pi /2} J_\nu(\pm \I z), \quad -\pi \le \arg(z) \le \pi/2.
\ee
The estimate \eqref{eq:asymp-J-infty} implies
\be
|J_{-l-1/2}(k)| \le k^{-1/2} (1+\OO(k^{-1}))
\ee
as $k\to \infty$. Therefore, there is  $C>0$ which depends only on $l$ and such that 
\be \label{besselest}
|\sqrt{k}J_{-l-1/2}(k)| \le C \left( \frac{1+k}{k}\right)^{l},\quad k>0.
\ee
By \eqref{besselest} we deduce 
\[
\frac{\sqrt{xy}}{2|\varepsilon+ \I t|} \left|\E^{-\frac{x^2 + y^2}{4(\varepsilon+ \I t)}} I_{-l-\frac{1}{2}}\Big(\frac{xy}{2(\varepsilon+\I t)}\Big)\right| \le C \sqrt{\frac{1}{|\varepsilon+\I t|}}\left| 1+\frac{2(\varepsilon+\I t)}{xy}\right|^{l},
\]
which is uniformly (wrt. $\varepsilon$) bounded on compact sets $K \subset \subset \R_+ \times \R_+$.
Thus we can apply dominated convergence and hence the claim follows.
\end{proof}

In particular, we immediately arrive at the following estimate.

\begin{corollary}\label{lem:est_pi/2}
Let $|l|<1/2$. Then there is a constant $C>0$ which depends only on $l$ and such that the inequality
\be\label{eq:3.10}
\big|[\E^{-\I tH_{\pi /2}}](x,y)\big| \le \frac{C}{\sqrt{2t}} \left( \frac{2t+xy}{xy}\right)^{l}
\ee
holds for all $x$, $y>0$ and $t> 0$.
\end{corollary}

\begin{proof}
Applying \eqref{besselest} to \eqref{eq:evkern2}, we arrive at \eqref{eq:3.10}. 
\end{proof}

\begin{remark}\label{rem:01}
For any fixed $x$ and $y \in \R_+$, we get from \eqref{eq:Jnu01} 
\be
\Big|\E^{-\I t H_{\pi/2}}(x,y)\Big| \sim \frac{\sqrt{xy}}{2t}   \left(\frac{xy}{4 t}\right)^{-l-1/2} = \frac{1}{t^{1/2 - l}} \left(\frac{xy}{2}\right)^{-l}
\ee
Moreover, in view of \eqref{eq:Jnu01} one can see that
\be
\Big|\E^{-\I t H_{\pi/2}}(x,y)\Big| \ge c_l\, {t^{l - 1/2}} \left(\frac{xy}{2}\right)^{-l},
\ee
whenever $xy < t$ with some constant $c_l>0$, which depends only on $l$.
\end{remark}

Now we are ready to prove our first main result.

\begin{proof}[Proof of Theorem \ref{th:pi/2}]
If $l\in (-1/2,0]$, then 
\[
\left( \frac{2t+xy}{xy}\right)^{l} \le 1
\]
for all $x$,$y>0$ and $t\ge 0$. This immediately implies \eqref{eq:03a}. 

Assume now that $l\in (0,1/2)$. Clearly,
\[
\frac{2t+xy}{xy} = 1+2\frac{t}{xy} \le 3t \max(x^{-1},1)\max(y^{-1},1) 
\]
for all $t\ge 1$ and $x$, $y>0$. Indeed, the latter follows from the weaker estimate
\[
\frac{t}{xy} \le t \max(x^{-1},1)\max(y^{-1},1),\quad t\ge 1,\ x,y>0, 
\]
which is equivalent to $1\le \max(x,1)\max(y,1)$ for all $x$, $y>0$. 
Therefore, 
\[
\left( \frac{2t+xy}{xy}\right)^{l} \le 3 t^{l} \max(x^{-l},1)\max(y^{-l},1),\quad t\ge 1,\ x,y>0,
\]
which proves \eqref{eq:03b}. Remark \ref{rem:01} shows that \eqref{eq:03b} is sharp.
\end{proof}

\section{Proof of Theorem \ref{th:no3}}\label{sec:IV}

Let us consider the following improper integrals:
\begin{align}\label{eq:I1}
I_1(t;x,y)&:=\sqrt{xy}\int_{\R_+} \E^{-\I t k^2} J_{l+\frac{1}{2}}(k x)J_{l+\frac{1}{2}}(k y)\im m_\alpha(k^2)\, k^{-2l}dk,\\
\label{eq:I2}
I_2(t;x,y)&:=\sqrt{xy}\int_{\R_+} \E^{-\I t k^2} J_{l+\frac{1}{2}}(k x)J_{-l-\frac{1}{2}}(k y)\im m_\alpha(k^2)\, kdk,\\
\label{eq:I3}
I_3(t;x,y)&:=\sqrt{xy}\int_{\R_+} \E^{-\I t k^2} J_{-l-\frac{1}{2}}(k x)J_{-l-\frac{1}{2}}(k y)\im m_\alpha(k^2)\, k^{2l+2}dk,
\end{align}
where $x$, $y>0$ and $t\neq 0$. Moreover, here and below we shall use the convention $\im m_\alpha(k^2) := \im m_\alpha(k^2+\I0) = \lim_{\varepsilon\downarrow 0}\im m_\alpha(k^2+\I \varepsilon)$ for all $k\in\R$.  Denote the corresponding integrand by $A_j$, that is, $I_j(t) = \int_{\R_+} \E^{-\I t k^2} A_j(k;x,y)dk$. Our aim is to use Lemma~\ref{lem:vC2} (plus the remarks after this lemma) and
hence we need to show that each $A_j$ belongs to the Wiener algebra $\cW(\R)$, that is, coincide with a function
which is the Fourier transform of a finite measure. 

We also need the following estimates, which follow from \eqref{eq:rho_a} 
\be\label{eq:m_inf}
\im m_\alpha(k^2) = \begin{cases} C_l^2 |k|^{2l+1}, & \alpha=0, \\[2mm] \frac{\cos^2(\pi l)}{C_l^2\sin^2(\alpha)} |k|^{-2l-1} +\OO(|k|^{-4l-2}), & \alpha\neq 0,\end{cases}\quad k\to \infty,
\ee
and 
\be\label{eq:m_at0}
\im m_\alpha(k^2) = \begin{cases} \frac{C_l^2}{\cos(\alpha)^2}|k|^{2l+1} + \OO(|k|^{4l+2}), & \alpha\neq\pi/2 , \\[2mm] C_l^{-2}\cos^2(\pi l) |k|^{-2l-1} , & \alpha = \pi/2,\end{cases}\quad k\to 0.
\ee

\subsection{The integral $I_1$}

Consider the function
\[
J(r) := \sqrt{r}\, J_{l+\frac{1}{2}}(r) = \frac{r^{l+1}}{2^{l+1/2}} \sum_{n=0}^\infty 
	\frac{(-r^2/4)^n}{n!\Gamma(\nu+n+1)}, \quad r\ge 0.
\]
Note that $J(r)\sim r^{l+1}$ as $r\to 0$ and $J(r)= \sqrt{\frac{2}{\pi}} \sin(r-\frac{l\pi}{2}) +O(r^{-1})$
as $r\to+\infty$ (see \eqref{eq:asymp-J-infty}). Moreover, $J'(r)\sim r^l$ as $r\to 0$ and $J'(r)= \sqrt{\frac{2}{\pi}} \cos(r-\frac{l\pi}{2}) +O(r^{-1})$ as $r\to+\infty$ (see \eqref{eq:a11}). In particular, $\ti{J}(r):=J(r)-\sqrt{\frac{2}{\pi}} \sin(r-\frac{l\pi}{2})$ is in $H^1(\R_+)$.
Moreover, we can define $J(r)$ for $r<0$ such that it is locally in $H^1$ and $J(r)=\sqrt{\frac{2}{\pi}} \sin(r-\frac{l\pi}{2})$
for $r<-1$. By construction we then have $\ti{J}\in H^1(\R)$ and thus $\ti{J}$ is the Fourier transform of an integrable
function (see Lemma \ref{lem:a.1}). Moreover, $\sin(r-\frac{l\pi}{2})$ is the Fourier transform of the sum of two Dirac delta measures and so
$J$ is the Fourier transform of a finite measure.
By scaling, the total variation of the measures corresponding to $J(k x)$ is independent of $x$.

Next consider the function
\begin{align*}
F(k) := \frac{\im m_\alpha(k^2)}{ |k|^{2l+1}} = \frac{C_l^2 }{(\cos(\alpha) - C_l^2 \sin(\alpha)\tan(\pi l) |k|^{2l+1})^2 + C_l^4\sin^2(\alpha) |k|^{4l+2}}.
\end{align*}
By Corollary \ref{lem:f_nu}, $F$ is in the Wiener algebra $\cW_0(\R)$.

Now it remains to note that 
\be
I_1(t)= \int_{\R_+} \E^{-\I t k^2} A_1(k^2;x,y)dk = 
\int_{\R_+} \E^{-\I t k^2} J(kx)J(ky) F(k) dk, 
\ee
and applying Lemma \ref{lem:vC2} we end up with the estimate
\be\label{eq:est_I1}
|I_1(t;x,y)| \le C t^{-1/2},\quad t>0,
\ee
with a positive constant $C>0$ independent of $x$, $y>0$.
\subsection{The integral $I_2$} 

Assume first that $l\in (0,1/2)$ and write
\[
A_2(k^2;x,y) = J(k x) Y(k y)  \frac{\chi_l(k)}{\chi_l(k y)} \frac{\im m_\alpha(k^2)}{\chi_l(k)},
\]
where
\[
J(r) =  \sqrt{r}\, J_{l+\frac{1}{2}}(r), 
\quad
Y(r) =  \chi_l(r) \sqrt{r}\, J_{-l-\frac{1}{2}}(r), 
\quad 
\chi_l(r)=\frac{|r|^l}{1+|r|^l}.
\]
The asymptotic behavior \eqref{eq:m_inf} and \eqref{eq:m_at0} of $\im m_\alpha$ shows that 
\[
M(k) = \frac{\im m_\alpha(k^2)}{\chi_l(k)} = \begin{cases} |k|^{1+l}, & k\to 0,\\ |k|^{-2l-1}, & |k|\to \infty,\end{cases}
\]
and hence $M\in H^1(\R)$, which implies that $M$ is in the Wiener algebra $\cW_0(\R)$.

We continue $J(r)$, $Y(r)$ to the region $r<0$ such that they are 
continuously differentiable and satisfy 
\[
J(r)=\sqrt{\frac{2}{\pi}} \sin\left(r-\frac{\pi l}{2}\right),\quad
Y(r)=\sqrt{\frac{2}{\pi}} \cos\left(r+\frac{\pi l}{2}\right),
\]
for $r<-1$. Then $\ti{J}(r):=J(r)-\sqrt{\frac{2}{\pi}} \sin(r-\frac{\pi l}{2})$ 
and $\ti{Y}(r):=Y(r)-\sqrt{\frac{2}{\pi}} \cos\left(r+\frac{\pi l}{2}\right)$ are in 
$H^1(\R)$. In fact, they are continuously differentiable and hence 
it suffices to look at their asymptotic behavior. For $r<-1$ they are 
zero and for $r>1$ they are $O(r^{-1})$ and their derivative is 
$O(r^{-1})$ as can be seen from the asymptotic behavior of Bessel 
functions (see Appendix \ref{sec:Bessel}). Hence both $J$ and $Y$ are Fourier 
transforms of finite measures. By scaling the total variation of the 
measures corresponding to $J(k x)$ and $Y(k y)$ are independent of $x$ and
$y$, respectively. 

It remains to consider the function $\chi_l(k)/\chi_l(ky)$. Observe that
\[
h_{y,l}(k) := 1 - \frac{\chi_l(k)}{\chi_l(k y)} = 1 - \frac{1+|ky|^l}{y^l + |ky|^l}=\frac{1-y^{-l}}{1+|k|^l} = (1-y^{-l})(1-\chi_l(k)).
\]
By Corollary \ref{lem:f_nu}, $1-\chi_l\in \cW_0(\R)$. Therefore, applying Lemma \ref{lem:vC2}, we obtain the following estimate
\be\label{eq:est_I2a}
|I_2(t;x,y)| \le Ct^{-1/2} \max(1,y^{-l}),\quad t>0,
\ee
whenever $l\in (0,1/2)$.

Consider now the remaining case $l\in (-1/2,0]$. Write
\[
A_2(k^2;x,y) = J(k x) Y(k y) \im m_\alpha(k^2),
\]
where
\[
J(r) =  \sqrt{r}\, J_{l+\frac{1}{2}}(r), 
\quad
Y(r) = \sqrt{r}\, J_{-l-\frac{1}{2}}(r).
\]
Noting that $Y(r) \sim r^{-l}$ as $r\to 0$ and using Lemma \ref{lem:a.1}, we can continue $J$ and $Y$ to the region $r<0$ such that both $J$ and $Y$ are Fourier 
transforms of finite measures. 

It remains to consider $\im m_\alpha(k^2)$ given by \eqref{eq:rho_a}. However, by Corollary \ref{lem:f_nu}, this function is in the Wiener algebra $\cW_0(\R)$ and hence applying Lemma \ref{lem:vC2}, we end up with the estimate
\be\label{eq:est_I2b}
|I_2(t;x,y)| \le Ct^{-1/2},\quad t>0,
\ee
whenever $l\in (-1/2,0]$.

\subsection{The integral $I_3$} 

Again let us consider two cases. Assume first that  $l\in (-1/2,0]$ and then write 
\[
A_3(k^2;x,y) = Y(k x) Y(k y)  \im m_\alpha(k^2)k^{2l+1},
\]
where
\[
Y(r) =   \sqrt{r}\, J_{-l-\frac{1}{2}}(r),\quad r>0.
\]
Notice that 
\[
|k|^{2l+1}\im m_\alpha(k^2)=\frac{C_l^2k ^{4l+2}}{(\cos(\alpha) - C_l^2 \sin(\alpha)\tan(\pi l)k^{2l+1})^2 + C_l^4\sin^2(\alpha) k^{4l+2}},
\]
which is the sum of a constant and a function of the form \eqref{eq:f_nu}, and hence it  belongs to the Wiener algebra $\cW(\R)$ by Corollary \ref{lem:f_nu}. Arguing as in the previous subsection and applying Lemma \ref{lem:vC2}, we arrive at the following estimate
\be\label{eq:est_I3a}
|I_3(t;x,y)| \le Ct^{-1/2},\quad t>0,
\ee
whenever $l\in (-1/2,0]$.

If $l\in(0,1/2)$, write
\[
A_3(k^2;x,y) = Y(k x) Y(k y)  \frac{\chi_l(k)}{\chi_l(k x)}\frac{\chi_l(k)}{\chi_l(k y)} \frac{\im m_\alpha(k^2)}{\chi_l^2(k)},
\]
where
\[
Y(r) =  \chi_l(r) \sqrt{r}\, J_{-l-\frac{1}{2}}(r), 
\quad 
\chi_l(r)=\frac{|r|^l}{1+|r|^l}.
\]
Notice that
\begin{align*}
M(k) :=& \frac{\im m_\alpha(k^2)|k|^{2l+1}}{\chi_l^2(k)}  \\
&= \frac{C_l^2|k| ^{2l+2} (1+k^l)^2}{(\cos(\alpha) - C_l^2 \sin(\alpha)\tan(\pi l)|k|^{2l+1})^2 + C_l^4\sin^2(\alpha) |k|^{4l+2}}
\end{align*}
Clearly, by Corollary \ref{lem:f_nu}, $M\in \cW(\R)$. Therefore, similar to  the previous subsection, we end up with the estimate
\be\label{eq:est_I3b}
|I_3(t;x,y)| \le Ct^{-1/2}\max(1,x^{-l})\max(1,y^{-l}),\quad t>0,
\ee
whenever $l\in (0,1/2)$.

\subsection{Proof of Theorem \ref{th:no3}} 

We begin with the representation of the integral kernel of the evolution group.

\begin{lemma}\label{lem:4.1}
Let $|l|<1/2$ and $\alpha\in [0,\pi)$. Then the evolution group $\E^{-\I tH_{\alpha}}P_c(H_\alpha)$ is an integral operator and its kernel is given by
\begin{align}
[\E^{-\I tH_\alpha}P_{c}(H_\alpha)](x,y)
        = \frac{2}{\pi} \int_{\R_+}
       \E^{-\I t k^2}\,\phi_\alpha(k^2,x) \phi_\alpha(k^2,y) \im m_\alpha(k^2) k \,dk, \label{rep}
\end{align}
where the integral is to be understood as an improper integral. 
\end{lemma}

\begin{proof}
By \eqref{eq:phi} and \eqref{eq:phi_a}, 
\begin{align*}
\phi_\alpha(k^2,x) &= \cos(\alpha) \phi(k^2,x) + \sin(\alpha) \theta(k^2,x) \\
&= \sqrt{\frac{\pi x}{2}} \left(C_l^{-1}\cos(\alpha) k^{-l-1/2}  
	  J_{l+\frac{1}{2}}(k x) + C_lk^{l+1/2}\frac{\sin(\alpha)}{\cos(\pi l)} 
		J_{-l-\frac{1}{2}}(k x)\right),
\end{align*}
and hence
\begin{align} \label{formula1}
\phi_\alpha(k^2,x)\phi_\alpha(k^2,y) &= \frac{\pi}{2}\sqrt{xy} \left(\frac{\cos^2(\alpha)}{C_l^2} k^{-2l-1} J_{l+\frac{1}{2}}(k x)J_{l+\frac{1}{2}}(k y) \right. \\
 &+ \frac{\sin(2\alpha)}{2\cos(\pi l)} (J_{l+\frac{1}{2}}(k x)J_{-l-\frac{1}{2}}(k y) + J_{-l-\frac{1}{2}}(k x)J_{l+\frac{1}{2}}(k y)) \\
 &\left. + C_l^2k^{2l+1} \frac{\sin^2(\alpha)}{\cos^2(\pi l)} 
		J_{-l-\frac{1}{2}}(k x)J_{-l-\frac{1}{2}}(k y)\right).
\end{align}
By our considerations in the previous subsections, we have
\[
\phi_\alpha(k^2,x) \phi_\alpha(k^2,y) \im m_\alpha(k^2) k \in \cW(\R)
\]
with norm uniformly bounded for $x,y$ restricted to any compact subset of $(0,\infty)$.
Moreover, we have $\E^{-\I (t - \I\eps) H_\alpha}P_{c}(H_\alpha) \to \E^{-\I tH_\alpha}P_{c}(H_\alpha)$
as $\eps\downarrow 0$ in the strong operator topology. By Lemma \ref{lem:c1}, $\E^{-\I (t - \I\eps) H_\alpha}P_{c}(H_\alpha)$ is an integral operator for all $\eps>0$ and, moreover, the kernel converges uniformly
on compact sets by Lemma \ref{lem:c.2}. Hence $\E^{-\I tH_\alpha}P_{c}(H_\alpha)$ is an
integral operator whose kernel is given by the limits of the kernels of the approximating operators,
that is, by \eqref{rep}.
\end{proof}

\begin{proof}[Proof of Theorem \ref{th:no3}] Combining \eqref{eq:est_I1}, \eqref{eq:est_I2a}, \eqref{eq:est_I2b}, \eqref{eq:est_I3a} and \eqref{eq:est_I3b}, we arrive at the following decay estimate for the kernel of the evolution group
\be\label{eq:est}
\Big| [\E^{-\I tH_\alpha}P_{c}(H_\alpha)](x,y)\Big| \le C t^{-1/2} \times \begin{cases} 1, & l \in (-1/2,0],\\ \max(1,x^{-l})\max(1,y^{-l}), & l\in (0,1/2).\end{cases}
\ee
This completes the proof of Theorem \ref{th:no3}.
\end{proof}

\appendix

\section{Bessel functions}\label{sec:Bessel}

Here we collect basic formulas and information on Bessel functions (see, e.g., \cite{dlmf, Wat}).
We start with the definition:
\begin{align}
J_{\nu}(z) 
= \left(\frac{z}{2}\right)^\nu \sum_{n=0}^\infty 
	\frac{(-z^2/4)^n}{n!\Gamma(\nu+n+1)}.\label{eq:Jnu01}
\end{align}
The asymptotic behavior as $|z|\to\infty$ is given by 
\begin{align}
J_{\nu}(z) 
	= \sqrt{\frac{2}{\pi z}}\left(\cos(z- \nu \pi/2 - \pi/4) 
	   + \E^{|\im z|}\OO(|z|^{-1})\right),
	   \quad |\arg z|<\pi. \label{eq:asymp-J-infty}
\end{align}
Noting that 
\be\label{eq:recurrence}
J_{\nu}'(z) = -J_{\nu+1}(z) + \frac{\nu}{z} J_{\nu}(z) =J_{\nu-1}(z) - \frac{\nu}{z} J_{\nu}(z) ,
\ee
one can show that the derivative of the reminder satisfies
\be\label{eq:a11}
\left(\sqrt{\frac{\pi z}{2}}J_{\nu}(z) - \cos(z-\frac{1}{2}\nu \pi - \frac{1}{4}\pi)\right)'= \E^{|\im z|}\OO(|z|^{-1}),\quad |z|\to \infty.
\ee

\section{The van der Corput Lemma and the Wiener algebra}\label{sec:vdCorput}

We will need the classical van der Corput lemma (see, e.g., \cite[page 334]{St}):

\begin{lemma}\label{lem:vC}
Consider the oscillatory integral
\begin{equation*}
I(t) = \int_a^b \E^{\I t k^2 + \I c k} A(k) dk.
\end{equation*}
If $A\in\mathrm{AC}(a,b)$, then
\begin{equation*}
\abs{I(t)} 
	\le C_2 \abs{t}^{-1/2}(\norm{A}_\infty + \norm{A'}_1), 
	\quad \abs{t}\ge 1,
\end{equation*}
where $C_2\le 2^{8/3}$ is a universal constant.
\end{lemma}

Note that we can apply the above result with $(a,b)=(-\infty,\infty)$
by considering the limit $(-a,a)\to(-\infty,\infty)$.

Our proof will be based on the following variant of the van der Corput 
lemma (see, e.g., \cite[Lemma A.2]{ktt}).

\begin{lemma}\label{lem:vC2}
Let $(a,b)\subseteq\R$ and consider the oscillatory integral
\[
I(t) = \int_a^b \E^{\I t k^2} A(k) dk.
\]
If $A \in \cW(\R)$, i.e., $A$ is the Fourier transform of a signed measure
\[
A(k) = \int_\R \E^{\I k p} d\alpha(p),
\]
then the above integral exists as an improper integral and satisfies
\begin{equation*}
\abs{I(t)} 
	\le C_2 \abs{t}^{-1/2} \norm{A}_\cW, 
	\quad \abs{t}>0.
\end{equation*}
where $\norm{A}_\cW:=\norm{\alpha}=\abs{\alpha}(\R)$ denotes the total variation of $\alpha$ and
$C_2$ is the constant from the van der Corput lemma.
\end{lemma}

In this respect we note that if $A_1$ and $A_2$ are two such functions, then 
(cf.\ p. 208 in \cite{bo})
\[
(A_1 A_2)(k)= \frac{1}{(2\pi)^2} \int_\R \E^{\I k p} d(\alpha_1*\alpha_2)(p)
\]
is associated with the convolution
\[
\alpha_1 * \alpha_2(\Omega) = \iint \id_\Omega(x+y) d\alpha_1(x) d\alpha_2(y),
\]
where $\id_\Omega$ is the indicator function of a set $\Omega$. 
Note that
\[
\|\alpha_1 * \alpha_2\| \le \|\alpha_1\| \|\alpha_2\|.
\]

Let $\cW_0(\R)$ be the Wiener algebra of functions $C(\R)$ which are Fourier transforms of $L^1$ functions,
\[
\cW_0(\R) = \Big\{f\in C(\R):\, f(k)=\int_{\R}\E^{\I kx} g(x)dx,\ g\in L^1(\R)\Big\}.
\]
Clearly, $\cW_0(\R)\subset \cW(\R)$. Moreover, by the Riemann--Lebesgue lemma, $f\in C_0(\R)$, that is, $f(k)\to 0$ as $k\to \infty$ if $f\in\cW_0(\R)$. A comprehensive survey of necessary and sufficient conditions for $f\in C(\R)$ to be in the Wiener algebras $\cW_0(\R)$ and $\cW(\R)$ can be found in \cite{lst12}, \cite{lt11}.  We need the following statements. 

\begin{lemma}\label{lem:a.1} 
If $f\in L^2(\R)$ is locally absolutely continuous and $f'\in L^p(\R)$ with $p\in (1,2]$, then $f$ is in the Wiener algebra $\cW_0(\R)$ and
\be\label{eq:embedding}
\|f\|_{\cW} \le C_p\big( \|f\|_{L^2(\R)} + \|f'\|_{L^p(\R)}\big),
\ee
where $C_p>0$ is a positive constant, which depends only on $p$.
\end{lemma}

\begin{proof}
Since the Fourier transform is unitary on $L^2(\R)$, it suffices to show that $\hat{f}\in L^1(\R)$.  First of all, the Cauchy--Schwarz inequality implies $\hat{f} \in L^1_{\loc}(\R)$ and, in particular,
\be
\int_{-1}^1 |\hat{f}(\lambda)| d\lambda \le \sqrt{2} \left(\int_{-1}^1 |\hat{f}(\lambda)|^{1/2} d\lambda\right)^2 \le \sqrt{2} \|f\|_{L^2(\R)}.
\ee

On the other hand, $f'\in L^p(\R)$ and hence the Hausdorff--Young inequality implies $\lambda \hat{f}(\lambda) \in L^q(\R)$ with $1/p+1/q=1$. Applying the H\"older inequality and then the Hausdorff--Young inequality once again, we get
\begin{align*}
\int_{|\lam|>1} |\hat{f}(\lam)| d\lam &\le 2\int_{|\lam|>1} \frac{1}{1+|\lam|} |\lam \hat{f}(\lam)| d\lam \\
&\le 2 \left( \int_{\R} \frac{1}{(1+|\lam|)^{p}}d\lam\right)^{1/p} \left( \int_{\R} |\lam \hat{f}(\lam)|^qd\lam\right)^{1/q} \le C_p' \|f'\|_{L^p(\R)},
\end{align*}
which completes the proof.
\end{proof}

\begin{remark}
The case $p=2$ is due to Beurling \cite[Theorem 5.3]{lst12}. A similar result was obtained by S.\ G.\ Samko. Namely, if $f\in L^1(\R) \cap AC_{\loc}(\R)$ is such that $f$, $f' \in L^p(\R)$ with some $p\in (1,2]$, then $f\in \cW_0(\R)$ (see Theorem 6.8 in \cite{lst12}).
\end{remark}

The next result is also due to Beurling (see, e.g., Theorem 5.4 in \cite{lst12}).

\begin{theorem}[Beurling]\label{th:beurling}
Let $f\in C_0(\R)$ be even and $f$, $f' \in AC_{\loc}(\R)$. If 
\be\label{eq:beurling}
C:=\int_{\R_+} k|f''(k)|dk <\infty,
\ee
then $f\in\cW_0(\R)$ and $\|f\|_\cW\le C$.
\end{theorem}

Consider the following functions, which appear in Section \ref{sec:IV}:
\begin{align}
\chi_l(k) &= \frac{|k|^l}{1+|k|^l},\quad l>0, \label{eq:chi} \\
f_{l,p}(k) &= \frac{|k|^p}{a+ b|k|^l+|k|^{2l}},\quad 2l>p\ge 0, \label{eq:f_nu}
\end{align}
 where $a$, $b\in\R$ are such that $a+ b|k|^p+|k|^{2p}>0$ for all $k\in\R$. 
As an immediate corollary of Beurling's result we get 

\begin{corollary}\label{lem:f_nu}
$\chi_l \in \cW(\R)$, $1-\chi_l\in\cW_0(\R)$, and  $ f_{l,p}\in \cW_0(\R)$.
\end{corollary}

\section{Integral kernels}\label{sec:kernel}

There are various criteria for operators in $L^p$ spaces to be integral operators (see, e.g., \cite{buk}). Below we present a simple sufficient condition on a function $K$ for $K(H)$ to be an integral operator, where $H$ is a one-dimensional Schr\"odinger operator. More precisely, 
let $H$ be a singular Schr\"odinger operator on $L^2(a,b)$ as in \cite{kst2} or \cite{kt2} with corresponding entire system of solutions
$\theta(z,x)$ and $\phi(z,x)$. Recall
\be
(H-z)^{-1} f(x) = \int_a^b G(z,x,y) f(y) dy,
\ee
where
\be\label{defgf}
G(z,x,y) = \begin{cases} \phi(z,x) \psi(z,y), & y\ge x,\\
\phi(z,y) \psi(z,x), & y\le x,\end{cases}
\ee
is the Green function of $H$ and $\psi(z,x)$ is the Weyl solution normalized by $W(\theta,\psi) =1$ (cf.\ \cite[Lem.~9.7]{tschroe}). We start with a simple lemma ensuring that a function
$K(H)$ is an integral operator. To this end recall that $K(H)$ is defined as $U^{-1} K U$ with $K$ the multiplication operator
in $L^2(\R,d\rho)$, $\rho$ the associates spectral measure, and $U: L^2(a,b)\to L^2(\R,d\rho)$ the spectral transformation
\be
(U f)(\lam) = \int_a^b \phi(\lam,x) f(x) dx.
\ee

\begin{lemma}\label{lem:c1}
Suppose $H$ is bounded from below and $|K(\lam)| \le C(1+|\lam|)^{-1}$ or otherwise $|K(\lam)| \le C(1+|\lam|)^{-2}$. Then $K(H)$ is an integral operator
\be
(K(H) f)(x) = \int_a^b K(x,y) f(y) dy,
\ee
with kernel
\be
K(x,y) = \int_\R K(\lam) \phi(\lam,x) \phi(\lam,y) d\rho(\lam).
\ee
In particular, $(1+|.|)^{-1/2} \phi(.,x) \in L^2(\R,d\rho)$ and $K(x,.)\in L^2(a,b)$ for every $x\in(a,b)$.
\end{lemma}

\begin{proof}
Note that (cf. \cite[Lemma 3.6]{kst2})
\[
(UG(z;x,.))(\lam) = \frac{\phi(\lam,x)}{z-\lam}.
\]
If $H$ is bounded from below then $G(z;x,.)$ is in the form domain of $H$ for fixed $x$ and every $z\in \C\setminus \sigma(H)$ (cf.\ \cite[(A.6)]{gst}) and we obtain from
\cite[Lemma~3.6]{kst2} that $(1+|\lam|)^{-1/2} \phi(\lam,x) \in L^2(\R,d\rho)$. In the general case we at least have  $G(z;x,.) \in L^2(a,b)$ and thus  $(1+|\lam|)^{-1} \phi(\lam,x) \in L^2(\R,d\rho)$.
Hence we can use Fubini's theorem to evaluate
\begin{align*}
K(H) f(x) &= U^{-1} K U f(x) = \int_\R \phi(x,\lam) K(\lam) \left( \int_a^b \phi(\lam,y) f(y) dy \right) d\rho(\lam)\\
&=\int_a^b K(x,y) f(y) dy .\qedhere
\end{align*}
\end{proof}

As a consequence we obtain that \eqref{rep} holds at least for $\im(t)<0$. To take the limit $\im(t)\to 0$ we need the following result
which follows from \cite[Lemma~3.1]{EKMT}.

\begin{lemma}\label{lem:c.2}
Consider the improper integral
\[
F(\eps) = \int_{-\infty}^\infty \E^{-\I(t+\I\eps) k^2} f(k) dk, \qquad \eps\le 0,
\]
where
\[
f(k) = \int_\R \E^{\I k p} d\alpha(p), \qquad |\alpha|(\R)<\infty.
\]
Then
\[
F(\eps) = \frac{1}{\sqrt{4\pi\I (t+\I\eps)}} \int_\R \E^{-\frac{p^2}{4(t+\I\eps)}} d\alpha(p).
\]
\end{lemma}

\end{document}